\documentclass[12pt]{amsart}
\usepackage{amsmath}
\usepackage{amssymb}
\usepackage{ifthen}
\usepackage{graphicx}
\usepackage{float}
\usepackage{cite}
\usepackage{amsfonts}
\usepackage{amscd}
\usepackage{amsxtra}
\usepackage{color}

\newtheorem{theorem}{Theorem}[section]
\newtheorem{lemma}[theorem]{Lemma}
\newtheorem{corollary}[theorem]{Corollary}

\theoremstyle{definition}
\newtheorem{remark}[theorem]{Remark}

\numberwithin{equation}{section}

\def\be{\begin{equation}}
\def\ee{\end{equation}}

\newcounter{alphabet}

\begin{document}
\title{Initial Successive coefficients for certain classes of univalent functions}
\author[V. Arora]{Vibhuti Arora}
\address{Department of Mathematics \\ National Institute of Technology Calicut\\
 India}
\email{vibhutiarora1991@gmail.com}


\begin{abstract}
We consider a
family of all analytic and univalent functions in the unit disk of the form $f(z)=z+a_2z^2+a_3z^3+\cdots$. The aim of this article is to investigate the bounds of the difference of moduli of initial successive coefficients, i.e. $\big | |a_{n+1}|-|a_n|\big |$ for $n=1,\,2$ and for some subclasses of analytic univalent functions. We found that all the estimations are sharp in nature by constructing some extremal functions.
\end{abstract}

\maketitle

\section{Introduction and main results}\label{Introduction}

Let $\mathcal{A}$ denote the class of functions $f$ of the form
\begin{equation}\label{S}
f(z)=a_1z+a_2z^2+a_3z^3+\cdots,
\end{equation}
with $a_1=1$, which are analytic in the unit disk $\mathbb{D}:=\{z\in \mathbb{C}:|z|<1\}$. Let $\mathcal{S}$ be the set all functions $f\in \mathcal{A}$ that are univalent in $\mathbb{D}$. For a general theory of univalent functions, we refer the classical books \cite{Dur83, Goo83, Pom75}.

The estimation of the difference of moduli of successive coefficients $\big||a_{n+1}|-|a_n|\big|$ is an
important problem in the study of univalent functions. It is well-known that the difference $|a_{n+1}|-|a_n|$ is bounded for $f\in \mathcal{S}$. Indeed, Hayman\cite{Hay63} proved $\big||a_{n+1}|-|a_n|\big |\le A$ for $f \in \mathcal{S}$, where $A \ge 1$
is an absolute constant. Pommerenke\cite{POM71} conjectured that $\big||a_{n+1}|-|a_n|\big| \le 1$ for the class of starlike function
which was proved by Leung\cite{Leu78}.  On the other hand, sharp bound is known only for $n=2$ (see \cite[Theorem~3.11]{Dur83}), namely
$$
-1\leq|a_3|-|a_2|\leq 1.029\ldots.
$$For convex functions, Li and Sugawa \cite{LS17}  investigated the sharp upper bound of $|a_{n+1}|-|a_n|$ for $n\ge 2$, and sharp lower bounds for $n=2,3$.
Several results are known in this direction. These observations are also addressed in the recent paper \cite{APS19}, where a bound for $\big||a_{n+1}|-|a_n|\big|$, $n\ge 2$, for the class of ${\gamma}$-spirallike functions of order $\alpha$ is obtained. However, the bound contains an unknown constant.

The successive coefficient problem was first studied for the class of univalent functions to prove the classical Bieberbach conjecture.
 Although, the Bieberbach conjecture has been proved after sixty nine years by De Branges in $1985$, the successive coefficient problem is still an open problem for several important class of functions including the whole classes of analytic univalent functions. Thus, the successive coefficient problem is still under consideration for the class of univalent functions and its subclasses even for some particular values of $n$.
In the present paper, the sharp bounds for $|a_2|-|a_1|$ and $|a_3|-|a_2|$ are studied when functions are 
$\gamma$-spirallike of order $\alpha$, $\gamma$-convex of order $\alpha$, and belonging to a well-known subclass of starlike functions. Note that, in $2021$, Sim and Thomas \cite{ST21} estimated the sharp bound for $|a_3|-|a_2|$ when $f$ is $\gamma$-spirallike function of order $\alpha$. However this paper provides an alternate proof for the above problem. Related work in the direction of the present investigation can also be found in \cite{Gri76,L18,PO19,ST20,SimT20,ST21}. 

\section{Preliminary results}
Let $\mathcal{P}$ denote the class of all analytic functions $p$ having positive real part in $\mathbb{D}$, with the form
$$
p(z)=1+c_1z+c_2z^2+\cdots.
$$
A member of $\mathcal{P}$ is called a {\em Carath\'{e}odory function}. It is known that $|c_n|\le 2$ for a function $p\in \mathcal{P}$ and for all $n\ge 1$ (see \cite{Dur83}).

Parametric representations of the coefficients are often more useful. Libera and Z\l{}otkiewicz \cite{Lib82,Lib83} derived the following parameterizations of possible
values of $c_2$ and $c_3$. 
\begin{lemma}\label{Pclass}
	Let $-2\le c_1 \le 2$ and $c_2,\,c_3\in \mathbb{C}$. Then there exists a function $p\in\mathcal{P}$ with
	$$
	p(z)=1+c_1z+c_2z^2+c_3z^3+\cdots
	$$
	if and only if 
	$$
	2c_2=c_1^2+(4-c_1^2)x
	$$
	and 
	$$
	4c_3=c_1^3+2(4-c_1^2)c_1 x^2-(4-c_1^2)c_1x^2+2(4-c_1^2)(1-|x|^2)y
	$$
	for some $x,y \in \mathbb{C}$ with $|x|\le 1$ and $|y|\le 1$.
\end{lemma}

A function $f\in \mathcal{A}$ is called {\em starlike} if $f(\mathbb{D})$ is a starlike domain with respect to origin. The
class of univalent starlike functions is denoted by $\mathcal{S}^*$. 
There is one natural generalization of starlike functions is $\gamma$-spirallike functions of order $\alpha$ which leads to a useful criterion for univalency. The family $\mathcal{S}_{\gamma}(\alpha)$ of 
{\em $\gamma$-spirallike functions of order $\alpha$} is defined
by
$$
\mathcal{S}_{\gamma}(\alpha)=\bigg\{ f\in \mathcal{A}:{\rm Re}\bigg(e^{-i\gamma}\cfrac{zf'(z)}{f(z)}\bigg)>\alpha \cos \gamma  \bigg\},
$$
where $0\le \alpha< 1$ and $-\pi/2< \gamma< \pi/2$. Each function in $\mathcal{S}_{\gamma}(\alpha)$ is univalent in $\mathbb{D}$ (see \cite{Lib67}). Functions in $\mathcal{S}_\gamma(0)$  are called \textit{$\gamma$-spirallike}, but they do not necessarily
belong to the starlike family $\mathcal{S}^*$. For example, the function
$f(z)=z(1-iz)^{i-1}$ is ${\pi/4}$-spirallike but $f \notin\mathcal{S}^*$.
The class $\mathcal{S}_\gamma (0)$ was introduced by ${\rm \check{S}}$pa${\rm\check{c}}$ek \cite{Spacek-33}
(see also \cite{Dur83}). Moreover, $\mathcal{S}_0 (\alpha)=:\mathcal{S}^*{(\alpha)}$ is the usual class of starlike functions of order $\alpha$,
and $\mathcal{S}^*(0)=\mathcal{S}^*$. Recall that the class  $\mathcal{S}_{\gamma}(\alpha)$, for $0\le \alpha<1$, is studied by several authors in different perspective (see, for instance \cite{KS20,Lib67}).

For simplification, we assume $\mu=e^{i\gamma}\cos\gamma,\,-\pi/2< \gamma< \pi/2$ and $a_1=1$ throughout the paper.
Next, recall the well known characterization of $\gamma$-spirallike functions of order $\alpha$. 
\begin{lemma}\label{iffspirallike}
For $-\pi/2< \gamma< \pi/2$ and $0\le \alpha <1$, a function $f\in \mathcal{S}_{\gamma}(\alpha)$ if and only if 
$$
f(z)=z\exp\bigg\{ (1-\alpha)\mu\int_{0}^{z}\cfrac{p(t)-1}{t}\,dt\bigg\}
$$
where
\begin{equation}\label{p(z)}
p(z)=\cfrac{1}{1-\alpha}\bigg\{ \cfrac{1}{\cos \gamma}\bigg(e^{-i\gamma} \cfrac{zf'(z)}{f(z)}+i\sin \gamma\bigg)-\alpha\bigg\}\in \mathcal{P}.
\end{equation}
\end{lemma}
We now state our main results which provide sharp bounds for $|a_2|-|a_1|$ and $|a_3|-|a_2|$ when the functions $f$ are $\gamma$-spirallike functions of order $\alpha$. Their proofs are given in Section \ref{proofs}.
\begin{theorem} \label{a2a1}
	Let $-\pi/2< \gamma< \pi/2$ and $0\le \alpha <1$. For every $f\in \mathcal{S}_{\gamma}(\alpha)$ of the form \eqref{S}, we have
	$$
-1\le|a_2|-|a_1|\le 2(1-\alpha) \cos\gamma-1.
$$
Equality holds on right-hand side only for the rotations of 
\begin{equation}\label{keq}
k_{\gamma,\alpha}(z)=\cfrac{z}{(1-z)^{2(1-\alpha)\mu}}
\end{equation}
and on the left-hand side equality holds only for the rotations of 
\begin{equation}\label{Equality}
h_{\gamma,\alpha}(z)=\cfrac{z}{(1-z^2)^{(1-\alpha)\mu}}.
\end{equation}
\end{theorem}
If we choose $\alpha=0$ in Theorem \ref{a2a1}, it produces the following result for the class of spirallike functions.
\begin{corollary}
Let $f \in \mathcal{S}_\gamma,\,-\pi/2< \gamma< \pi/2$, given by \eqref{S}. Then we have
$$
-1\le|a_2|-|a_1|\le 2 \cos\gamma-1.
$$
The right-hand inequality becomes equality only for the rotations of $k_{\gamma,0}(z)=z/(1-z)^{2\mu}$ and left-hand inequality becomes equality only for the rotations of $h_{\gamma,0}(z)=z/(1-z^2)^{\mu}$.
\end{corollary}
\begin{theorem}\label{a3a2}
 Let $-\pi/2< \gamma< \pi/2$ and $0\le \alpha <1$. For every $f\in \mathcal{S}_{\gamma}(\alpha)$ in the form \eqref{S}, we have
$$
\cfrac{-2(1-\alpha)\cos\gamma}{\sqrt{1+T(\alpha,\gamma)}}\le|a_3|-|a_2|\le (1-\alpha) \cos\gamma,
$$
where 
\begin{equation}\label{eqT}
T(\alpha, \gamma)=\sqrt{1+4(1-\alpha)(2-\alpha)\cos^2\gamma}.
\end{equation}
 Equality holds on right-hand side only for the rotations of $h_{\gamma,\alpha}$ given by \eqref{Equality} and on the left-hand side only for the rotations of
 \begin{equation}\label{eqf}
 f_{\gamma,\alpha}(z)=\cfrac{z}{[(1-\epsilon_1z)^{\gamma_1}(1-\epsilon_2z)^{\gamma_2}]^{2(1-\alpha)\mu}},
 \end{equation}
 where
 \begin{align}\label{eqe1e2}
 |\epsilon_1|=|\epsilon_2|&=1,\,\epsilon_1\ne\epsilon_2,\,\gamma_1,\gamma_2>0,\,\gamma_1+\gamma_2=1,\nonumber\\
 \gamma_1\epsilon_1+\gamma_2\epsilon_2&=\cfrac{c}{2},\, \mbox{ and } \gamma_1\epsilon_1^2+\gamma_2\epsilon_2^2=\cfrac{(c^2+(4-c^2)x)}{4}
 \end{align}
 with $c=2/\sqrt{T(\alpha,\gamma)+1}$ and $x=-(1+2(1-\alpha)\cos^2\gamma+i(1-\alpha)\sin(2\gamma))/T(\alpha,\gamma)$.
\end{theorem}
If we put $\alpha=0$ in Theorem \ref{a3a2}, then we obtain the following result:
\begin{corollary}\label{coroa2a3}
Let $f \in \mathcal{S}_\gamma,\,-\pi/2< \gamma< \pi/2$, given by \eqref{S}. Then we have 
$$
\cfrac{-2\cos\gamma}{\sqrt{1+T(0,\gamma)}}\le|a_3|-|a_2|\le \cos\gamma.
$$
Equality holds on the right-hand side only for the rotations of $h_\gamma(z)=z/(1-z^2)^{\mu}$ and on the left-hand side equality holds for the rotations of $ f_{\gamma,0}$ defined by \eqref{eqf}.
\end{corollary}
It is appropriate to remark that Corollary \ref{coroa2a3} coincides with \cite[Theorem 1.4]{L18}. Also note that for 
$\alpha=0$ and $\gamma=0$, Theorems \ref{a2a1} and \ref{a3a2} extend the result of Leung \cite{Leu78} from starlike to $\gamma$-spirallike functions of order $\alpha$.

We consider another family of functions that includes
the class of convex functions as a proper subfamily. For
$-\pi/2<\gamma<\pi/2$ and $0\le \alpha<1$, we say that $f$ belonging to the family ${\mathcal C}_\gamma (\alpha) $ of {\em $\gamma$-convex of order $\alpha$}
provided $f\in {\mathcal A}$ is locally univalent in $\mathbb{D}$
and  $zf'(z)$ belongs to $\mathcal{S}_\gamma (\alpha)$, i.e.
$$
{\mathcal C}_\gamma (\alpha)=\left\{f\in \mathcal{A}:{\rm Re } \bigg\{ e^{-i\gamma }\left ( 1+\frac{zf''(z)}{f'(z)}\right )\bigg\}>\alpha \cos \gamma\right\}.
$$
We may set ${\mathcal C}_\gamma (0) =:{\mathcal C}_\gamma$  and observe that the class
${\mathcal C}_0(\alpha)=:{\mathcal C}(\alpha)$ consists of the normalized
convex functions of order $\alpha$. The setting ${\mathcal C}(0)=:{\mathcal C}$ is the usual class of convex functions. For general values of $\gamma $
$(|\gamma|<\pi/2)$, a function in ${\mathcal C}_\gamma$ need not be
univalent in $\mathbb{D}$. For example, the function
$f(z)=i(1-z)^i-i$ is known to belong to
${\mathcal C}_{\pi/4}\backslash {\mathcal S}$. In \cite{Pfaltzgraff}, Pfaltzgraff has shown that
$f\in\mathcal{C}_{\gamma}$ is univalent whenever $0<\cos \gamma\leq
1/2$. On the other hand,  in \cite{SinghChic-77} it was also shown that
a function $f$ in ${\mathcal C}_\gamma$ which satisfies $f''(0)=0$ is
univalent for all real values of $\gamma$ with $|\gamma|<\pi /2$.
For a general reference about these special classes we refer to
\cite{Goo83}.

Now let us recall the classical Alexander theorem which gives the close analytic connection between the convex and starlike functions. Analogues to this, it can be verified that $f$ belongs to $\mathcal{C}_{\gamma}(\alpha)$ if and only if $zf'$ belongs to $\mathcal{S}_{\gamma}(\alpha)$. This will
be used in the sequel.

In the next theorem, we will discuss about the sharp bounds for $|a_2|-|a_1|$ and $|a_3|-|a_2|$ when the functions $f$ run over the class $\mathcal{C}_{\gamma}(\alpha)$.
\begin{theorem}\label{convexa1a2}
Let $-\pi/2< \gamma< \pi/2$ and $0\le \alpha <1$. For every $f\in \mathcal{C}_{\gamma}(\alpha)$ be of the form \eqref{S}, we have
$$
-1\le|a_2|-|a_1|\le (1-\alpha)\cos \gamma-1.
$$
The right-hand side inequality becomes equality only for the rotations of
\begin{equation}\label{leq}
l_{\gamma,\alpha}(z)=\cfrac{1}{2(1-\alpha)\mu-1}\bigg[\cfrac{1}{(1-z)^{2(1-\alpha)\mu-1}}-1\bigg].
\end{equation}
The left-hand side inequality becomes equality only for the rotations of
\begin{equation}\label{qcequality}
q_{\gamma,\alpha}(z)=\int_{0}^{z}\cfrac{1}{(1-t^2)^{(1-\alpha)\mu}}\,dt.
\end{equation}
\end{theorem}
Note that $k_{\gamma,\alpha}(z)=zl_{\gamma,\alpha}'(z)$ and $h_{\gamma,\alpha}(z)=zq_{\gamma,\alpha}'(z)$.
For the special case $\gamma=0$, we get the following result:
\begin{corollary}
	Let $0\le \alpha <1$. For every $f\in \mathcal{C}(\alpha)$ be of the form \eqref{S}, we have
$$
-1\le|a_2|-|a_1|\le -\alpha. 
$$
Equality holds on the right-hand side only for the rotations of functions $l_{0,\alpha}$ given by \eqref{leq} and on the left-hand side only for the rotations of $q_{0,\alpha}$ given by \eqref{qcequality}.
\end{corollary}

\begin{theorem}\label{convexa2a3}
Let $-\pi/2< \gamma< \pi/2$ and $0\le \alpha <1$. For every $f\in \mathcal{C}_{\gamma}(\alpha)$ of the form \eqref{S}, we have
$$
\cfrac{-(1-\alpha)\cos\gamma}{\sqrt{1+T(\alpha,\gamma)}}\le|a_3|-|a_2|\le \cfrac{(1-\alpha)\cos \gamma}{3}.
$$
 Equality holds on the right-hand side only for the rotations of the functions $q_{\gamma,\alpha}$ given by \eqref{qcequality} and on the left-hand side only for the rotations of
\begin{equation}\label{geq}
g_{\gamma,\alpha}(z)=\int_{0}^{z}\cfrac{1}{[(1-\epsilon_1t)^{\gamma_1}(1-\epsilon_2t)^{\gamma_2}]^{2(1-\alpha)\mu}}\, dt,
\end{equation}
where $T(\alpha,\gamma)$ is given by \eqref{eqT} and $\gamma_1,\gamma_2,\epsilon_1,\epsilon_2,x$ satisfy \eqref{eqe1e2} with $c=2/\sqrt{T(\alpha,\gamma)+1}$ and $x=-(1+2(1-\alpha)\cos^2\gamma+i(1-\alpha)\sin(2\gamma))/T(\alpha,\gamma)$.
\end{theorem}
The following corollary immediately follows, by choosing $\gamma=0$, from Theorem \ref{convexa2a3} for the class of convex functions of order $\alpha$:
\begin{corollary}
Let $0\le \alpha <1$. For every $f\in \mathcal{C}(\alpha)$ of the form \eqref{S}, we have
$$
\cfrac{-(1-\alpha)}{\sqrt{1+T(\alpha,0)}}\le|a_3|-|a_2|\le \cfrac{(1-\alpha)}{3},
$$
where $T(\alpha,\gamma)$ is given by \eqref{eqT}. Equality holds on the right-hand side only for the rotations of the function $q_{0,\alpha}$ given by \eqref{qcequality} and on the left-hand side only for the rotations of $g_{0,\alpha}$ defined by \eqref{geq}.
\end{corollary}
We remark that Theorem \ref{convexa2a3}, for $\gamma=0$ and $\alpha=0$, is obtained by Li and Sugawa \cite{LS17} for the class of convex functions.

Let $\mathcal{LU}$ denote the subclass of $\mathcal{A}$ consisting of all {\em locally univalent functions};
namely, $\mathcal{LU}=\{f\in\mathcal{A}:f'(z)\ne0,z\in\mathbb{D}\}$. A family $\mathcal{G}(\lambda)$, $\lambda>0$, of functions $f\in \mathcal{LU}$ is defined by
$$
{\mathcal G}(\lambda)=\left\{f\in \mathcal{LU}:{\rm Re } \left ( 1+\frac{zf''(z)}{f'(z)}\right )<1+\cfrac{\lambda}{2}\right\}.
$$
The class $\mathcal{G}:=\mathcal{G}(1)$ was first introduced by Ozaki \cite{Ozaki41} and proved the inclusion relation $\mathcal{G}\subset \mathcal{S}$.
The Taylor coefficient problem for the class $\mathcal{G}(\lambda)$, $0<\lambda\leq 1$, is discussed in \cite{Obradovic13}. Recently, the radius of convexity of the functions in the class $\mathcal{G}(\lambda)$, $\lambda>0$, is obtained in \cite{Shankey20}. This class, with special choices of the parameter $\lambda$, has also been considered by many researchers in the literature for different purposes; see for instance \cite{AS18,ponnusamy95,ponnusamy96,ponnusamy07,Shankey20}.

The family $\mathcal{G}(\lambda)$ can be characterized in terms of the Carath\'{e}odory function as follows:
\begin{lemma}\label{giff}
For $0<\lambda \le1$, a function $f\in \mathcal{G}(\lambda)$ if and only if 
$$
f'(z)=\exp\bigg\{ \cfrac{\lambda}{2}\int_{0}^{z}\cfrac{p(t)-1}{t}\,dt\bigg\},
$$
where
\begin{equation}\label{gp(z)}
p(z)=\cfrac{1}{\lambda} \bigg(\lambda- \cfrac{2zf''(z)}{f'(z)}\bigg)\in \mathcal{P}.
\end{equation}
\end{lemma}
\begin{proof}
Let $f\in \mathcal{G}(\lambda)$. Then from the definition of the class $\mathcal{G}(\lambda)$, we have
$$
{\rm Re } \left ( 1+\frac{zf''(z)}{f'(z)}\right )<1+\cfrac{\lambda}{2}
$$
and hence we may rewrite the above inequality as
$$
{\rm Re } \left ( \lambda-2\frac{zf''(z)}{f'(z)}\right )>0.
$$
Consider the function
$$
p(z)=\cfrac{1}{\lambda}\bigg(\lambda-\cfrac{2zf''(z)}{f'(z)}\bigg).
$$
Observe that $p(0)=1,\,{\rm Re}(p(z))>0$, and hence $p(z)$ is clearly a Carath\'{e}odory function. The function $p$ simplifies to 
$$
\cfrac{p(z)-1}{2z}=\cfrac{f''(z)}{\lambda f'(z)}.
$$
Usual integration from $0$ to $z$ along any path leads to
$$
\lambda\int_{0}^{z}\cfrac{p(t)-1}{2t}\,dt=\int_{0}^{z}\cfrac{f''(t)}{ f'(t)}\,dt=\log(f'(z))
$$
or equivalently, it finally brings to the form
$$
f'(z)=\exp\bigg\{ \cfrac{\lambda}{2}\int_{0}^{z}\cfrac{p(t)-1}{t}\,dt\bigg\}.
$$
This completes the proof of Lemma \ref{giff}.
\end{proof}

Our next results are related to the sharp bounds for $|a_2|-|a_1|$ and $|a_3|-|a_2|$ when the functions $f$ are belonging to the class $\mathcal{G}(\lambda)$.
\begin{theorem}\label{ga2a3}
Let $0<\lambda\le 1$. If $f\in \mathcal{G}(\lambda)$ given by \eqref{S}, then
$$
-1\le|a_2|-|a_1|\le\cfrac{\lambda}{2}-1.
$$
Equality holds on the right-hand side for the rotations of 
\begin{equation}\label{Geq}
G_{\lambda}(z)=\cfrac{(1+z)^{1+\lambda}-1}{\lambda+1}.
\end{equation}
Equality holds on the left-hand side for the rotations of 
\begin{equation}\label{Heq}
H_{\lambda}(z)=\int_{0}^{z}(1-t^2)^{\lambda/2}dt.
\end{equation}
\end{theorem}
\begin{theorem}\label{thmg}
For $0\le \lambda <1$, let every function $f\in \mathcal{G}(\lambda)$ be defined by \eqref{S}. Then we have
$$
|a_3|-|a_2|\le \cfrac{\lambda}{6}.
$$
The inequality becomes equality only for the rotations of $H_{\lambda}$ defined by $\eqref{Heq}$. Furthermore,
$$
|a_3|-|a_2|\ge\left \{
\begin{array}{ll}
\cfrac{\lambda(4\lambda-17)}{24(2-\lambda)}, & {\mbox{ for }} 0<\lambda\le 1/2,
\\[5mm]
-\cfrac{\lambda(\lambda+2)}{6}, & {\mbox{ for }} 1/2\le\lambda\le1.
\end{array}
\right.
$$
 The inequality becomes equality only for the rotations of
\begin{equation*}
F_{\lambda}(z)=\int_{0}^{z}[(1-\epsilon_1t)^{\gamma_1}(1-\epsilon_2t)^{\gamma_2})]^\lambda,
\end{equation*}
where $\gamma_1,\gamma_2,\epsilon_1,\epsilon_2$ satisfy \eqref{eqe1e2} with $x=-1$, $c=3/(2-\lambda)$ for  $0<\lambda\le 1/2$, and $c=2$ for $1/2\le\lambda\le1$.
\end{theorem}
\section{Proof of the main results}\label{proofs}
This section is devoted to the detailed discussion on our proof of our main results.
\subsection{Proof of Theorem \ref{a2a1}}
	Let $f(z)=z+\sum_{n=2}^{\infty}a_nz^n\in \mathcal{S}_{\gamma}(\alpha)$. Then by the definition, we may consider  $p(z)=1+c_1z+c_2z^2+\cdots\in \mathcal{P}$ of the form \eqref{p(z)} which is equivalent to writing 
	$$
	((1-\alpha)p(z)+\alpha)\cos \gamma-i \sin \gamma=e^{-i\gamma}\cfrac{zf'(z)}{f(z)}.
	$$
	By using the Taylor representations of the functions $f$ and $p$, and comparing the coefficients of $z^n\, (n=1,2)$ both the sides, we get
	\begin{equation}\label{a2a3}
	a_2=(1-\alpha)\mu c_1 \text{ and } 2a_3=(1-\alpha)^2\mu^2 c_1^2+(1-\alpha)\mu c_2.
	\end{equation}
	So,
	$$
	|a_2|-|a_1|=|a_2|-1=(1-\alpha) \cos\gamma |c_1|-1
	\le 2(1-\alpha) \cos\gamma-1,\\
	$$
	where the last inequality comes by using $|c_n|\le 2$ for $n\ge 1$. 
	For the equality, let us consider the function $f=	k_{\gamma,\alpha}$ given by \eqref{keq} for which $a_2=2(1-\alpha)\mu$.
Then it is a simple exercise to see that
	$$
	{\rm Re}\bigg(e^{-i\gamma}\cfrac{zk'_{\gamma,\alpha}(z)}{k_{\gamma,\alpha}(z)}\bigg)=\cos \gamma\bigg[1+2(1-\alpha){\rm Re}\bigg(\cfrac{z}{1-z}\bigg)\bigg]>\alpha\cos\gamma,
	$$
	from which we can easily conclude that $k_{\gamma,\alpha}\in\mathcal{S}_{\gamma}(\alpha)$ and $|a_2|-|a_1|=2(1-\alpha)\cos\gamma-1$.
	
	On the other hand, 
	$$
	|a_1|-|a_2|=1-(1-\alpha)\cos \gamma|c_1|\le 1.
	$$
	Consider the function $f=h_{\gamma,\alpha}$ given by \eqref{Equality}. In this case $a_1=1$ and $a_2=0$.
A simple calculation shows that
	$$
	{\rm Re}\bigg(e^{-i\gamma}\cfrac{zh'_{\gamma,\alpha}(z)}{h_{\gamma,\alpha}(z)}\bigg)=\cos \gamma\bigg[1+2(1-\alpha){\rm Re}\bigg(\cfrac{z^2}{1-z^2}\bigg)\bigg]>\alpha\cos\gamma.
	$$
	Thus, the left-hand side equality holds for the function $f=h_{\gamma,\alpha}\in \mathcal{S}_{\gamma}(\alpha)$.
	This completes the proof.\hfill{$\Box$}
\subsection{Proof of Theorem \ref{a3a2}} Let $f\in \mathcal{S}_{\gamma}(\alpha)$. Then from equation \eqref{a2a3} we get
	\begin{align*}
	|a_3|-|a_2|&=\bigg|\cfrac{(1-\alpha)^2\mu^2 c_1^2+(1-\alpha)\mu c_2}{2}\bigg|-|(1-\alpha)\mu c_1|\\
	&=\cfrac{(1-\alpha)|\mu|}{2}\bigg[\big|{(1-\alpha)\mu c_1^2+c_2}\big|-2| c_1|\bigg].
		\end{align*}
	As $|a_3|-|a_2|$ is invariant under rotation, to simplify the calculation we assume that $c_1=c\in [0,2]$. Therefore, by Lemma \ref{Pclass}, for some $x\in \overline{\mathbb{D}}$ we have
	\begin{equation}\label{eqa3a2x}
	|a_3|-|a_2|=\cfrac{(1-\alpha)|\mu|}{4}\bigg[| c^2+(4-c^2)x+2(1-\alpha)\mu c^2|-4 c\bigg]=\cfrac{(1-\alpha)|\mu|}{4}\bigg[\psi(x,c)-4 c\bigg]
	\end{equation}
	with $\psi(x,c):=| c^2+(4-c^2)x+2(1-\alpha)\mu c^2|$. By letting $x=re^{i\theta}$, we compute
	\begin{align}\label{equality}
\psi(x,c)&=\big| c^2+(4-c^2)r(\cos\theta+i \sin \theta)+2(1-\alpha)\cos\gamma(\cos \gamma+i\sin \gamma) c^2\big|\nonumber\\
	&=\Big(\sqrt{P\cos\theta+Q\sin \theta+R}-4c\Big)
	\end{align}
		where $P=2r(4-c^2)c^2(1+2(1-\alpha)\cos^2\gamma),\,Q=2r(1-\alpha)(4-c^2)c^2\sin 2\gamma$, and $R=c^4+(4-c^2)^2r^2+4c^4(1-\alpha)^2\cos^2\gamma+4c^4(1-\alpha)\cos^2\gamma$. Clearly, $\sqrt{P^2+Q^2}=2rc^2(4-c^2)T(\alpha,\gamma)$ with $T(\alpha, \gamma)=\sqrt{1+4(1-\alpha)(2-\alpha)\cos^2\gamma}$. Next, we use the following well-known inequality
	\begin{equation}\label{PQR}
	-	\sqrt{{P^2+Q^2}}\le{P\cos\theta+Q\sin \theta}\le 	\sqrt{P^2+Q^2}
	\end{equation} 
	 to obtain 
	\begin{align*}
	\psi(x,c)&\le \sqrt{\sqrt{P^2+Q^2}+R} \\	&=\sqrt{2rc^2(4-c^2)T(\alpha,\gamma)+(4-c^2)^2r^2+c^4T^2(\alpha,\gamma)}\\
	&=c^2T(\alpha,\gamma)+(4-c^2)r\le c^2T(\alpha,\gamma)+4-c^2.
	\end{align*}
		 Since $\gamma\in (-\pi/2,\pi/2)$ and $\alpha\in [0,1)$, it is easy to check that $T(\alpha,\gamma)\le 3$. Hence,
\begin{equation}\label{bound}
	\psi(x,c)\le 2c^2+4.
\end{equation}
By substituting \eqref{bound} into \eqref{eqa3a2x}, we get	
$$
	|a_3|-|a_2|\le \cfrac{(1-\alpha)|\mu|}{4}(2c^2+4-4c)\le \cos\gamma (1-\alpha),
	$$
as required. Recall that, in the previous theorem we have already proved that $h_{\gamma,\alpha}$ defined by \eqref{Equality} belongs to $\mathcal{S}_{\gamma}(\alpha)$. It is evident that the equality holds for the function $h_{\gamma,\alpha}$ in which the coefficient of $z^2$ is $0$ and $z^3$ is $(1-\alpha)\mu$. Thus, the right-hand equality of the theorem has been proved.

	On the other hand, by \eqref{equality} and \eqref{PQR} we have
	\begin{align}
	\psi(x,c)&\ge \sqrt{-\sqrt{P^2+Q^2}+R}\nonumber \\	&=\sqrt{-2rc^2(4-c^2)T(\alpha,\gamma)+(4-c^2)^2r^2+c^4T^2(\alpha,\gamma)}\nonumber\\
	&=|c^2T(\alpha,\gamma)-(4-c^2)r|\ge |c^2T(\alpha,\gamma)-(4-c^2)|.\label{ge}
	\end{align}
Here, the first equality occurs if $\cos \theta=-(1+2(1-\alpha)\cos^2\gamma)/T(\alpha,\gamma)$ and $\sin \theta=(1-\alpha)\sin(2\gamma)/T(\alpha,\gamma)$. Also, the last equality holds when $r=1$. The equation \eqref{eqa3a2x} and the inequality \eqref{ge} together lead to
\begin{equation}\label{eqt}
|a_3|-|a_2|\ge\cfrac{(1-\alpha)|\mu|}{4}\Big(|c^2T(\alpha,\gamma)-(4-c^2)|-4c\Big).
\end{equation}
We may have the following two cases:\\

	\medskip\noindent
	{\bf Case 1:} Let $c^2(T(\alpha,\gamma)+1)-4\ge 0$. Then \eqref{eqt} becomes
	$$
	|a_3|-|a_2|\ge \cfrac{(1-\alpha)|\mu|}{4}\Big(c^2(T(\alpha,\gamma)+1)-4-4c\Big).
	$$
	It is easy to check that $c^2(T(\alpha,\gamma)+1)-4-4c$ is an increasing function of $c$ in the interval $[2/\sqrt{1+T(\alpha,\gamma)}, 2]$.
	
	\medskip\noindent
	{\bf Case 2:} Let $c^2(T(\alpha,\gamma)+1)-4\le 0$. Then
	$$
	|a_3|-|a_2|\ge \cfrac{(1-\alpha)|\mu|}{4}\Big(-c^2(T(\alpha,\gamma)+1)+4-4c\Big).
	$$
	Since $T(\alpha,\gamma)\ge 1$, we know  $-c^2(T(\alpha,\gamma)+1)+4-4c$ is a decreasing function of $c$ in the interval $[0,2/\sqrt{1+T(\alpha,\gamma)}]$.
	
	\medskip\noindent
	Hence, the minimum attains in \eqref{eqt} for $c=2/\sqrt{1+T(\alpha,\gamma)}$. It follows that
	$$
	|a_3|-|a_2|\ge  \cfrac{-2(1-\alpha)\cos\gamma}{\sqrt{1+T(\alpha,\gamma)}}.
	$$
	
	We now proceed to prove the left-hand side equality part. Choose $$
	x=-\frac{1+2(1-\alpha)\cos^2\gamma+i(1-\alpha)\sin(2\gamma)}{T(\alpha,\gamma)} \mbox{ and } c=\frac{2}{\sqrt{1+T(\alpha,\gamma)}}.
	$$
	Note that $4-c^2\ge0$ and $|x|=1$. Thus by making use of the Carath\'{e}odory-Toeplitz theorem (see \cite{GS58},\cite{Tsuji}), we obtain $p$ is of the following form
	\begin{equation}\label{p(z)form}
	p(z)=\gamma_1\cfrac{1+\epsilon_1z}{1-\epsilon_1z}+\gamma_2\cfrac{1+\epsilon_2z}{1-\epsilon_2z}\in \mathcal{P},
	\end{equation}
	where $\gamma_1,\gamma_2,\epsilon_1,\epsilon_2$ satisfy \eqref{eqe1e2}.

Now Lemma \ref{iffspirallike} gives
	$$
	f_{\gamma,\alpha}(z)=\cfrac{z}{(1-\epsilon_1z)^{2(1-\alpha)\mu\gamma_1}(1-\epsilon_2z)^{2(1-\alpha)\mu\gamma_2}},
	$$
which belongs to $\mathcal{S}_{\gamma}(\alpha)$. This can indeed be verified by using the series expansion of $f_{\gamma,\alpha}$ to get
$$
	a_2=\cfrac{2(1-\alpha)\mu}{\sqrt{1+T(\alpha,\gamma)}}  \text{ and }a_3=\cfrac{(1-\alpha)\mu}{1+T(\alpha,\gamma)}\bigg[2(1-\alpha)\mu +1+xT(\alpha,\gamma)\bigg].
	$$
	This gives that
	$$
	|a_3|-|a_2|=\cfrac{-2(1-\alpha)\cos\gamma}{\sqrt{1+T(\alpha,\gamma)}},
	$$
	which yields the desired result.
\hfill{$\Box$}
\subsection{Proof of Theorem \ref{convexa1a2}}
	Suppose $f\in \mathcal{C}_{\gamma}(\alpha)$. Then there exist a function $p\in \mathcal{P}$ such that
$$
p(z)=\cfrac{1}{1-\alpha}\bigg\{ \cfrac{1}{\cos \gamma}\bigg[e^{-i\gamma} \bigg(1+\cfrac{zf''(z)}{f'(z)}\bigg)+i\sin \gamma\bigg]-\alpha\bigg\},
$$
or equivalently
\begin{equation}\label{eqcp(z)}
((1-\alpha)p(z)+\alpha)\cos \gamma-i \sin \gamma=e^{-i\gamma} \bigg(1+\cfrac{zf''(z)}{f'(z)}\bigg).
\end{equation}
We write 
$$
f(z)=\sum_{n=1}^{\infty}a_nz^n \mbox{ and }p(z)=1+\sum_{n=1}^{\infty}p_nz^n.
$$
Equating the coefficients of $z^n$ on both the sides of the equation \eqref{eqcp(z)} for $n=1,2$, we obtain
\begin{equation}\label{eqa2a3}
a_1=1,\,2a_2=(1-\alpha)\mu c_1 \mbox{  and  } 6a_3=(1-\alpha)^2\mu^2 c_1^2+(1-\alpha)\mu c_2.
\end{equation}
Now we compute and estimate
$$
|a_2|-|a_1|=\bigg|\cfrac{(1-\alpha)\mu c_1}{2}\bigg|-1\le(1-\alpha)\cos \gamma-1.
$$
The last inequality holds since $|c_1|\le 2$. It is easy to check that the equality holds for the function $l_{\gamma,\alpha}$ is given by \eqref{leq} and which satisfies $zl'_{\gamma,\alpha}=k_{\gamma,\alpha}$, where $k_{\gamma,\alpha}\in \mathcal{S}_{\gamma}(\alpha)$ is given by \eqref{keq}. Thus $l_{\gamma,\alpha}\in\mathcal{C}_{\gamma}(\alpha)$ and the coefficient of $z^2$ in $l_{\gamma,\alpha}$ is $(1-\alpha)\mu$.

 Secondly, we estimate the lower bound for $|a_1|-|a_2|=1-(1-\alpha)|\mu c_1|/2\le 1$. For the sharpness, let us consider the function $q_{\gamma,\alpha}$ satisfying $zq'_{\gamma,\alpha}=h_{\gamma,\alpha}$, where $h_{\gamma,\alpha}$ is defined by \eqref{Equality}. Since $h_{\gamma,\alpha}\in \mathcal{S}_{\gamma}(\alpha)$, it concludes that $q_{\gamma,\alpha}\in \mathcal{C}_{\gamma}(\alpha)$. Also,
 \begin{align*}
 q_{\gamma,\alpha}(z)&=\int_{0}^{z}\cfrac{1}{(1-t^2)^{(1-\alpha)\mu}}\,dt=\int_{0}^{z}\bigg[1+(1-\alpha)\mu t^2+\cfrac{(1-\alpha)\mu((1-\alpha)\mu+1)}{2}\,t^4+\cdots\bigg]\\
 &=z+\cfrac{(1-\alpha)\mu}{3}\,z^3+\cfrac{(1-\alpha)\mu((1-\alpha)\mu+1)}{10}\,z^5+\cdots.
 \end{align*}
 For this function, we have
 \begin{equation}\label{gcoeff}
 a_1=1,\,a_2=0,\mbox{ and } a_3=\cfrac{(1-\alpha)\mu}{3},
 \end{equation}
 and hence $|a_2|-|a_1|=-1$. This completes the proof.
\hfill{$\Box$}

\subsection{Proof of Theorem \ref{convexa2a3}}To prove this theorem, we use the similar technique that is adopted in Theorem \ref{a3a2}. Let $f\in \mathcal{C}_{\gamma}(\alpha)$. Then by means of equation \eqref{eqa2a3}, we see that
\begin{align*}
|a_3|-|a_2|&=\bigg|\cfrac{(1-\alpha)^2\mu^2 c_1^2+(1-\alpha)\mu c_2}{6}\bigg|-\bigg|\cfrac{(1-\alpha)\mu c_1}{2}\bigg|\\
&=\cfrac{(1-\alpha)|\mu|}{6}\big[|\mu(1-\alpha)c_1^2+c_2|-3|c_1|\big].
\end{align*}
	As $|a_3|-|a_2|$ is invariant under rotations, to simplify the calculation we assume that $c_1=c\in [0,2]$. Thus, Lemma \ref{Pclass} yields
\begin{equation}\label{ca3a2x}
|a_3|-|a_2|=\cfrac{(1-\alpha)|\mu|}{12}\big[\psi(x,c)-6c\big]
\end{equation}
for some $x\in\overline{\mathbb{D}}$. By using equation \eqref{bound} we derive the desired inequality
$$
|a_3|-|a_2|\le\cfrac{(1-\alpha)|\mu|}{12}\Big(2c^2+4-6c\Big)\le \cfrac{(1-\alpha)\cos \gamma}{3}.
$$
It is clear from \eqref{gcoeff} that the equality holds for the function $q_{\gamma,\alpha}\in\mathcal{C}_{\gamma}(\alpha)$ given by \eqref{qcequality}.

We next find the lower bound of $|a_3|-|a_2|$. The inequality \eqref{ge} and equation \eqref{ca3a2x} together lead to 
\begin{equation}\label{ceqa3a2}
|a_3|-|a_2|\ge\cfrac{(1-\alpha)|\mu|}{12}\Big(|c^2T(\alpha,\gamma)-(4-c^2)|-6c\Big).
\end{equation}
Next we consider the following two cases in order to complete the proof:

\medskip\noindent
{\bf Case 1:} Let $c^2(T(\alpha,\gamma)+1)-4\ge 0$. Then
$$
|a_3|-|a_2|\ge \cfrac{(1-\alpha)|\mu|}{12}\Big(c^2(T(\alpha,\gamma)+1)-4-6c\Big).
$$
It is easy to check that $c^2(T(\alpha,\gamma)+1)-4-6c$ is an increasing function of $c$ in the interval $[2/\sqrt{1+T(\alpha,\gamma)}, 2]$.

\medskip\noindent
{\bf Case 2:} Let $c^2(T(\alpha,\gamma)+1)-4\le 0$. Then
$$
|a_3|-|a_2|\ge \cfrac{(1-\alpha)|\mu|}{12}\Big(-c^2(T(\alpha,\gamma)+1)+4-6c\Big).
$$
Since $T(\alpha,\gamma)\ge 1$, we know  $-c^2(T(\alpha,\gamma)+1)+4-6c$ is a decreasing function of $c$ in the interval $[0,2/\sqrt{1+T(\alpha,\gamma)}]$.

\medskip\noindent
Therefore, the minimum in \eqref{ceqa3a2} is attained at $c=2/\sqrt{1+T(\alpha,\gamma)}$, which implies that
	$$
|a_3|-|a_2|\ge  \cfrac{-(1-\alpha)\cos\gamma}{\sqrt{1+T(\alpha,\gamma)}}.
$$
For the equality, we consider the function $g_{\gamma,\alpha}$ defined as $zg'_{\gamma,\alpha}=f_{\gamma,\alpha}$, where $f_{\gamma,\alpha}\in \mathcal{S}_{\gamma}(\alpha)$ is given by \eqref{eqf}. Thus, $g_{\gamma,\alpha}\in \mathcal{C}_{\gamma}(\alpha)$ with the representation \eqref{geq}. The series expansion of $g_{\gamma,\alpha}$ has the form
\begin{equation*}
g_{\gamma,\alpha}(z)=z+\cfrac{\mu(1-\alpha)}{\sqrt{1+T(\alpha,\gamma)}}\,z^2+\cfrac{(1-\alpha)\mu[2(1-\alpha)\mu+1+T(\alpha,\gamma)x]}{3(1+T(\alpha,\gamma))}\,z^3+\cdots,
\end{equation*}
completing the proof.
\hfill{$\Box$}	
\subsection{Proof of Theorem \ref{ga2a3}} Let $f\in \mathcal{G}(\lambda)$. Then there exists a function $p(z)=1+c_1z+c_2z^2+\cdots\in\mathcal{P}$ satisfying \eqref{gp(z)}. Therefore, it is equivalent to write
$$
\lambda f'(z)p(z)=\lambda f'(z)-2zf''(z).
$$
After writing  $f$ and $p$ in the series form and by comparing the coefficients of $z$ and $z^2$ in the above equation, we obtain the relations
\begin{equation}\label{eqqga2a3}
a_2=-\cfrac{\lambda c_1}{4} \mbox{ and } a_3=\cfrac{\lambda^2 c_1^2-2\lambda c_2}{24}.
\end{equation}
Consider
$$
|a_2|-|a_1|=\cfrac{\lambda|c_1|}{4}-1\le\cfrac{\lambda}{2}-1. $$
To prove the equality part, consider the function $G_{\lambda}$ provided by \eqref{Geq} for which $a_2=\lambda/2$.
An easy computation yields
$$
{\rm Re } \left ( 1+\frac{zG_{\lambda}''(z)}{G_{\lambda}'(z)}\right )=1+\lambda{\rm Re } \left ( \frac{z}{1+z}\right )<1+\cfrac{\lambda}{2},
$$
which shows that $G_{\lambda} \in \mathcal{G}(\lambda)$.

Next, $|a_1|-|a_2|=1-\lambda|c_1|/4\le 1$. As the function $H_{\lambda}$, defined by \eqref{Heq}, satisfies 
$$
{\rm Re } \left ( 1+\frac{zH_{\lambda}''(z)}{H_{\lambda}'(z)}\right )=1-\lambda{\rm Re } \left ( \frac{z^2}{1-z^2}\right )<1+\cfrac{\lambda}{2}
$$
and
\begin{equation}\label{eqHa3a2}
a_1=1,\,a_2=0, \mbox{ and } a_3=\cfrac{\lambda}{6}\,,
\end{equation}
 it is clear that the left-hand equality holds for $H_{\lambda}\in\mathcal{G}(\lambda)$. Proof of the theorem is now completed. \hfill{$\Box$}

\subsection{Proof of Theorem \ref{thmg}}
We use the equation \eqref{eqqga2a3} to compute
$$
|a_3|-|a_2|=\bigg|\cfrac{\lambda^2 c_1^2-2\lambda c_2}{24}\bigg|-\bigg|\cfrac{\lambda c_1}{4}\bigg|=\cfrac{\lambda}{24}\big(|\lambda c_1^2-2 c_2|-6|c_1|\big).
$$
We can check that the functional $|a_3|-|a_2|$ are rotationally invariant, so we assume $c_1=c\in [0,1]$. Also, we note that
$$
|a_3|-|a_2|=\cfrac{\lambda}{24}\big(| c^2(1-\lambda)+x(4-c^2)|-6c\big)
$$
follows from Lemma \ref{Pclass} for some $x\in \overline{\mathbb{D}}$ and therefore by substituting $x=re^{i\theta}$ we deduce that
\begin{align}\label{eqga3a2}
|a_3|-|a_2|&=\cfrac{\lambda}{24}\big(\sqrt{(1-\lambda)^2c^4+r^2(4-c^2)^2+2rc^2(1-\lambda)(4-c^2)\cos \theta}-6c\big)\\
&\le\cfrac{\lambda}{24}\big(((1-\lambda)c^2+r(4-c^2))-6c\big)\nonumber\\
&\le\cfrac{\lambda}{24}\big((1-\lambda)c^2+(4-c^2)-6c\big)= \cfrac{\lambda}{24}(4-\lambda c^2-6c)\le\cfrac{\lambda}{6}\nonumber.
\end{align}
The desired inequality thus follows. As we noted in the previous theorem, the function $H_{\lambda}$ given by \eqref{Heq} belongs to the class $\mathcal{G}(\lambda)$. Therefore, from \eqref{eqHa3a2} we conclude now that the equality occurs for $H_{\lambda}$.

We now proceed to prove the left-hand side inequality. By using the equality \eqref{eqga3a2} and the inequality $\cos \theta\ge -1$ we have that
\begin{align}\label{eqmin}
|a_3|-|a_2|&\ge \cfrac{\lambda}{24}\big(|(1-\lambda)c^2-r(4-c^2)|-6c\big)\nonumber\\
&\ge \cfrac{\lambda}{24}\big(|(1-\lambda)c^2-(4-c^2)|-6c\big)=\cfrac{\lambda}{24}\big(|c^2(2-\lambda)-4|-6c\big).
\end{align}
The last equality holds when $r=1$.
Next we consider the following two cases in order to complete the proof:

\medskip\noindent
{\bf Case 1:} Let $c^2(2-\lambda)-4\le 0$. Then
$$
|a_3|-|a_2|\ge\cfrac{\lambda}{24}\big(-c^2(2-\lambda)+4-6c\big).
$$
It is easy to verify that $-c^2(2-\lambda)+4-6c$ is a decreasing function of $c$ in the interval $[0,2/\sqrt{2-\lambda}]$.

\medskip\noindent
{\bf Case 2:} Let $c^2(2-\lambda)-4\ge 0$. Then
$$
|a_3|-|a_2|\ge\cfrac{\lambda}{24}\big(c^2(2-\lambda)-4-6c\big).
$$
It is easy to verify that $\phi(c):=c^2(2-\lambda)-4-6c$ is a decreasing function of $c$ in the interval $[2/\sqrt{2-\lambda}, 3/(2-\lambda)]$. Thus,

\medskip\noindent
{\bf (a)} If $0<\lambda\le1/2$, then $\phi(c)$ is increasing in the interval $[3/(2-\lambda), 2]$.\\
{\bf (b)} If $1/2\le\lambda\le1$, then $\phi(c)$ is decreasing in the interval $[2/\sqrt{2-\lambda}, 2]$.\\

\medskip\noindent
Therefore, for $0<\lambda\le1/2$ the minimum in \eqref{eqmin} is attained at $c=3/(2-\lambda)$, which implies that
$$
|a_3|-|a_2|\ge\cfrac{\lambda(4\lambda-17)}{24(2-\lambda)},
$$
and for $1/2\le\lambda\le1$, \eqref{eqmin} takes its minimum value at $c=2$ which implies that
$$
|a_3|-|a_2|\ge-\cfrac{\lambda(\lambda+2)}{6}.
$$

Note here that equalities hold simultaneously above for a suitable choice of $x$ and $c$. Choose $x=-1$ and 
$$
c=\left \{
\begin{array}{ll}
\cfrac{3}{2-\lambda}, & {\mbox{ for }} 0<\lambda\le 1/2,
\\[5mm]
2, & {\mbox{ for }} 1/2\le\lambda\le1.
\end{array}
\right.
$$
Then by the Carath\'{e}odory-Toeplitz theorem, it follows that $p$ has the form \eqref{p(z)form} satisfying \eqref{eqe1e2}.
Now Lemma \ref{giff} together with $p$ gives the required form of the extremal function
$$
F_{\lambda}(z)=\int_{0}^{z}[(1-\epsilon_1t)^{\gamma_1}(1-\epsilon_2t)^{\gamma_2})]^\lambda.
$$
This completes the proof of Theorem \ref{ga2a3}.\hfill{$\Box$}

\begin{remark}
	In this paper, we found the sharp bounds of $\big | |a_{n+1}|-|a_n|\big |$ for the class $\mathcal{S}_{\gamma}(\alpha),\,\mathcal{C}_{\gamma}(\alpha)$, and $\mathcal{G}(\lambda)$ only for $n=1,2$.
For the remaining positive values of $n$
(i.e. for $n\ge 3$) this problem is still open.
Investigation for a complete solution to this problem may lead to new techniques in this development.
\end{remark}

\bigskip
\noindent
{\bf Declaration of competing interest.} The author declares that there is no conflict of interest regarding the publication of this paper.

\section*{Acknowledgments}
I would like to thank my Ph.D. supervisor Dr. Swadesh Kumar
Sahoo for his helpful remarks and suggestions.

\end{document}